\documentclass[twoside,11pt]{article}

\usepackage[preprint]{jmlr2e}
\usepackage{algpseudocode}
\usepackage{algorithm}
\usepackage{float}
\usepackage{amsmath}
\usepackage{amssymb}
\usepackage{optidef}

\ShortHeadings{Safeguarding Mechanism for L-BFGS Inverse Hessian Approximation Matrix}{Li}
\firstpageno{1}

\begin{document}

\title{On the Condition Number Upper Bound of the L-BFGS Inverse Hessian Approximation Matrix with a Two-Sided Geometric Envelope Safeguarding Mechanism}

\author{\name Don Li \email don@pdx.edu \\
       \addr Department of Mathematics \& Statistics\\
       Portland State University \\
       Portland, OR 97201, USA}

\editor{N/A}

\maketitle

\begin{abstract}
The limited-memory BFGS (L-BFGS) algorithm is a cornerstone of large-scale optimization due to its linear memory and computational costs. However, in ill-conditioned or non-convex landscapes, the implicit inverse Hessian approximation can suffer from an exploding condition number, leading to numerical instability and degraded convergence. To address this, we propose Two-Sided L-BFGS, a safeguarded variant that dynamically constrains the condition number of the inverse Hessian operator via a two-sided geometric envelope. Moreover, we show that Two-Sided L-BFGS preserves accumulated curvature information and maintains standard $O(mn)$ memory and per-iteration time complexities. We prove that this geometric envelope yields a uniform bound on the condition number of every inverse Hessian approximation generated by the algorithm. By tracking the algebraic evolution of the extreme eigenvalues through $m$ consecutive quasi-Newton updates starting from a scaled identity matrix, the resulting bound is expressed explicitly as a function of the memory depth, problem dimension, and envelope hyperparameters. Moreover, we show that Two-Sided L-BFGS preserves asymptotic global convergence in non-convex regimes under standard smoothness and strong Wolfe line-search assumptions, matching the theoretical guarantees of L-BFGS variants utilizing the Li-Fukushima cautious update rule. Numerical experiments on high-dimensional optimization problems demonstrate that the proposed method maintains well-conditioned inverse Hessian approximations and improves robustness and convergence behavior on ill-conditioned benchmarks.
\end{abstract}


\section{Introduction}

Quasi-Newton methods are widely used in a plethora of different large-scale optimization settings, ranging from highly non-convex deep learning architectures [\cite{RM19}] to numerical weather prediction and climate modeling [\cite{DN03}]. Such methods balance improved convergence rates with lower computational costs via approximations of the inverse Hessian matrix using only the gradient. Recall that, in general form, Newton's method employs the update rule, for parameter $x_k$,

\begin{equation}
x_{k+1} = x_k - \eta[\nabla^2f(x_k)]^{-1} \nabla f(x_k),
\end{equation}

\noindent
where $[\nabla^2f(x_k)]^{-1}$ is the inverse Hessian, $\nabla f(x_k)$ is the gradient, and $\eta$ is the \textit{learning rate} or \textit{step size}. $\nabla^2f(x_k)$ provides crucial second-order curvature information and guarantees descent for $\nabla^2f(x_k) \succ 0$, but is computationally infeasible in large-scale settings. The \textit{Broyden, Fletcher, Goldfarb, and Shanno} (BFGS) algorithm was introduced as a Quasi-Newton method to address this computational bottleneck with respect to finding $[\nabla^2f(x_k)]^{-1}$. Quasi-Newton methods replace the exact $[\nabla^2f(x_k)]^{-1}$ with an approximation matrix $H_k$, where $B_k$ is the approximation matrix of $\nabla^2f(x_k)$. As alluded to earlier, as a Quasi-Newton method, BFGS constructs $B_k$ from the gradient, where Hessian approximation updates must satisfy the \textit{secant condition}, i.e.,

\begin{equation}
B_{k+1}s_k = y_k,
\end{equation}

\noindent
which we derive as follows:

\begin{theorem}[Secant Condition for Hessian Approximation Updates]
Let $f: \mathbb{R}^n \rightarrow \mathbb{R}$ be a twice continuously differentiable function. Let $B_k$ be an approximation of the Hessian matrix $\nabla^2f(x_k)$ with step vector $s_k := x_{k+1}-x_k$ and gradient difference $y_k := \nabla f(x_{k+1}) - \nabla f(x_k)$ for consecutive parameter iterates $x_k$, $x_{k+1}$. Then any BFGS Hessian approximation update must satisfy

\[
B_{k+1}s_k = y_k.
\]
\end{theorem}

\begin{proof}
Perform a second-order Taylor series expansion of $f(x_k)$. Evaluated at the subsequent parameter update $x_{k+1}$, this yields

\[
f(x_k) \approx f(x_{k+1}) + \nabla f(x_{k+1})^T (x_k - x_{k+1}) + \frac{1}{2}(x_k - x_{k+1})^T \nabla^2f(x_{k+1})(x_k - x_{k+1}).
\]

\noindent
Differentiating with respect to $x_k$ in turn yields

\[
\nabla f(x_k) \approx \nabla f(x_{k+1}) + \nabla^2f(x_{k+1})(x_k - x_{k+1}), 
\]

\noindent
or $\nabla f(x_k) - \nabla f(x_{k+1}) \approx \nabla^2f(x_{k+1})(x_k - x_{k+1})$. Since $s_k : = x_{k+1}-x_k$ and $y_k := \nabla f(x_{k+1}) - \nabla f(x_{k})$, we define $B_{k+1}$ such that it satisfies

\[
B_{k+1}s_k = y_k.
\]
\end{proof}

\noindent
An analogous secant condition likewise must hold for BFGS inverse Hessian approximation updates.

\begin{corollary}[Secant Condition for Inverse Hessian]
Let $H_k$ denote the inverse Hessian approximation matrix. Then any BFGS inverse Hessian approximation update must satisfy

\[
H_{k+1}y_k = s_k.
\]
\end{corollary}

\begin{proof}
Recall via second-order Taylor series expansion of $f(x_k)$ that we obtain $\nabla f(x_k) \approx \nabla f(x_{k+1}) + \nabla^2f(x_{k+1})(x_k - x_{k+1})$, or $\nabla f(x_k) - \nabla f(x_{k+1}) \approx \nabla^2f(x_{k+1})(x_k - x_{k+1})$. This implies

\[
(\nabla f(x_k) - \nabla f(x_{k+1}))[\nabla^2 f(x_{k+1})]^{-1} \approx x_k - x_{k+1}.
\]

\noindent
Let $H_{k+1} := [\nabla^2f(x_{k+1})]^{-1}$. Since $s_k : = x_{k+1}-x_k$ and $y_k := \nabla f(x_{k+1}) - \nabla f(x_{k})$, we define $H_{k+1}$ such that it satisfies

\[
H_{k+1}y_k = s_k.
\]
\end{proof}

\noindent
While the secant condition allows for the approximation of the Hessian from the gradient, BFGS still requires a line search subroutine to perform parameter updates. As \cite{NW06} note, the parameter update takes the form

\begin{equation}
x_{k+1} = x_k + \alpha_k p_k,
\end{equation}

\noindent
where $\alpha_k > 0$ is the \textit{step length} and $p_k$ is the \textit{search direction}. We show that the search direction takes the form

\begin{equation}
p_k = -B_k^{-1}\nabla f_k,
\end{equation}

\noindent
where $B_k = B_k^T$ and $\text{det}(B_k) \neq 0$.

\begin{theorem}[Search Direction for Line Search Methods]
Let $f: \mathbb{R}^n \rightarrow \mathbb{R}$ be a twice continuously differentiable function, and let $B_k$ be the Quasi-Newton Hessian approximation matrix where $B_k = B_k^T$ and $\text{det}(B_k) \neq 0$. Then

\[
p_k = -B_k^{-1}\nabla f_k.
\]
\end{theorem}

\begin{proof}
Recall the update rule

\[
x_{k+1} = x_k + \alpha_kp_k
\]

\noindent
for step length $\alpha_k > 0$. Perform a second-order Taylor series expansion of $f(x_k + p_k)$ centered at $x_k$, which yields

\[
f(x_k + p_k) \approx f(x_k) + \nabla f(x_k)^T p_k + \frac{1}{2}p_k^T \nabla^2f(x_k)p_k.
\]

\noindent
Optimal $p_k$ corresponds to $\frac{\partial}{\partial p_k} [f(x_k) + \nabla f(x_k)^T p_k + \frac{1}{2}p_k^T \nabla^2f(x_k)p_k]=0$. This implies

\[
\nabla f(x_k) + \nabla^2f(x_k)p_k = 0,
\]

\noindent
which in turn is equivalent to 

\[
p_k = -[\nabla^2f(x_k)]^{-1}\nabla f(x_k).
\]

\noindent
But recall that, for Quasi-Newton methods, we approximate $-[\nabla^2f(x_k)]^{-1}$ with $B_k^{-1}$, which gives us $p_k = -B_k^{-1}\nabla f_k$.
\end{proof}

\noindent
If $p_k = -B_k^{-1}\nabla f_k$ and $B_k \succ 0$, then $p_k^T \nabla f(x_k) = -\nabla f(x_k)^T B_k^{-1} \nabla f(x_k) < 0$, so $p_k$ is a descent direction [\cite{NW06}]. Moreover, recall $H_k$ is the Quasi-Newton approximation of $[\nabla^2f(x_k)]^{-1}$, so the descent direction is equivalently

\begin{equation}
p_k = -H_k\nabla f(x_k).
\end{equation}

\noindent
Note that, with $p_k = -B_k^{-1}\nabla f(x_k)$, the parameter update rule from Eqn. (3) takes the form

\begin{equation}
x_{k+1} = x_k - \alpha_k B_k^{-1} \nabla f(x_k).
\end{equation}

\noindent
This is the parameter update rule for BFGS, so BFGS is clearly a quasi-Newton method with $B_k^{-1}$ as the approximation of $[\nabla^2 f(x_k)]^{-1}$ and $\alpha_k$ is a dynamical step size corresponding to $\eta$. Note that $B_k^{-1}$ is governed by the update rule

\begin{equation}
B_{k+1}^{-1} = (I - \rho_k s_k y_k^T)B_k^{-1}(I - \rho_k y_k s_k^T) + \rho_k s_k s_k^T,
\end{equation}

\noindent
where $\rho_k = \frac{1}{y_k^T s_k}$ is a scaling factor that constraints admissible $B_{k+1}^{-1}$ such that $B_{k+1}^{-1}$ satisfies the secant condition, and $I$ is the identity matrix. We present the BFGS pseudocode:

\begin{algorithm}[H]
\caption{Broyden-Fletcher-Goldfarb-Shanno (BFGS)}
\begin{algorithmic}

\Require Intialization of parameter vector $x_0$, tolerance $\epsilon > 0$, and symmetric $B_0 \succ 0$ (typically $B_0 = I$)

\While{$||\nabla f(x_k)|| > \epsilon$}
    \State $p_k = -B_k^{-1}\nabla f(x_k)$

    \State Find admissible $\alpha_k$

    \State $x_{k+1} = x_k + \alpha_k p_k$

    \State $s_k = x_{k+1} - x_k$

    \State $y_k = \nabla f(x_{k+1}) - \nabla f(x_k)$

    \State $\rho_k = \frac{1}{y_k^T s_k}$

    \State $B_{k+1}^{-1} = (I - \rho_k s_k y_k^T)B_k^{-1}(I - \rho_k y_k s_k^T) + \rho_k s_k s_k^T$

    \State $k += 1$
\EndWhile

\noindent
\Return $x_{k+1}$
\end{algorithmic}
\end{algorithm}

\noindent
Admissible $\alpha_k$ are values of $\alpha_k$ such that they satisfy the \textit{Wolfe conditions}, the first of which is known as the \textit{Armijo condition}, i.e., 

\begin{equation}
f(x_k + \alpha_k p_k) \leq f(x_k) + c_1\alpha_k \nabla f(x_k)^Tp_k,
\end{equation}

\noindent
for $c_1 \in (0,1)$. The Armijo condition guarantees sufficient decrease in $f$ that is proportional to both $\alpha_k$ and the directional derivative $\nabla f(x_k)^Tp_k$ [\cite{NW06}]. The second of the Wolfe conditions is the \textit{curvature condition}, i.e.,

\begin{equation}
\nabla f(x_k + \alpha_k p_k)^T p_k \geq c_2 \nabla f(x_k)^T p_k,
\end{equation}

\noindent
for $c_2 \in (c_1, 1)$. The curvature condition precludes insufficiently large $\alpha_k$ [\cite{NW06}]. Suppose $\phi(\alpha) := f(x_k + \alpha p_k)$. Desirable $\alpha_k$ is that which minimizes $\phi(\alpha)$. $\alpha_k$ may satisfy inequalities (8) and (9) but not necessarily be in the neighborhood of a minimizer of $\phi(\alpha)$ [\cite{NW06}]. To ensure this, we can enforce the condition that 

\begin{equation}
|\nabla f(x_k + \alpha_k p_k)^T| \leq c_2 |\nabla f(x_k)^T p_k|.
\end{equation}

\noindent
Inequalities (8) and (10) together constitute the \textit{strong Wolfe conditions}.

\subsection{Limited-Memory BFGS (L-BFGS)}

BFGS is shown to have space complexity $O(n^2)$, an improvement over that of Newton's method, which is $O(n^3)$. However, BFGS starts to become computationally infeasible around roughly $n > 10^3$. In these higher-dimensional settings, \textit{limited-memory BFGS} (L-BFGS) is the preferred algorithm. Recall the BFGS update rule

\begin{equation}
x_{k+1} = x_k - \alpha_k H_k \nabla f(x_k).
\end{equation}

\noindent
Even for $H_k$ satisfying the secant condition, storing the full $H_k$ is computationally infeasible for sufficiently large $n$. The premise of L-BFGS is to circumvent this by only storing a limited $m$ number of vector pairs $(s_i, y_i)$ as a proxy for $H_k \nabla f_k$, where typically $3 \leq m \leq 20$ in practice [\cite{NW06}]. Once the memory buffer contains $m$ number of vector pairs, the oldest $(s_i, y_i)$ is replaced with the newest $(s_k, y_k)$, with $m$ number of vector pairs kept in memory for the duration of the execution of L-BFGS. L-BFGS is typically implemented in two-loop recursion form, as introduced in \cite{LN89}, with search direction $r = H_k \nabla f(x_k)$ as output. Recall the inverse Hessian $H_k$ update rule,

\begin{equation}
H_{k+1} = (I - \rho_k s_k y_k^T)H_k(I - \rho_k y_k s_k^T) + \rho_k s_k s_k^T,
\end{equation}

\noindent
with scaling factor  $\rho_k = \frac{1}{y_k^T s_k}$. 

\begin{algorithm}[H]
\caption{Limited-Memory Broyden-Fletcher-Goldfarb-Shanno (L-BFGS)}
\begin{algorithmic}

\Require Initialization of current iteration $k$, current gradient $q = \nabla f(x_k)$, and history vectors $s_i$, $y_i$ for $i = k-1$ down to $\text{max}(0, k-m)$

\For{$i = k-1$ \textbf{down to} $k-m$} \Comment{Backward Pass}
    \State $\alpha_i \gets \rho_i s_i^T q$
    \State $q \gets q - \alpha_i y_i$
\EndFor

\State $r \gets \gamma_k q$

\For{$i = k-m$ \textbf{up to} $k-1$} \Comment{Forward Pass}
    \State $\beta \gets \rho_i y_i^T r$
    \State $r \gets r + s_i(\alpha_i -\beta)$
\EndFor

\State \Return $r = H_k \nabla f(x_k)$
\end{algorithmic}
\end{algorithm}

\noindent
L-BFGS has space complexity $O(mn)$, thus improving over the $O(n^2)$ space complexity of BFGS. 

\subsection{Oren-Spedicato Scaling}

Recall that, in L-BFGS, after the backward pass and before the forward pass, we perform the update $r \gets \gamma_k q$ with gradient placeholder $q = \nabla f(x_k)$. $\gamma_k$ is the \textit{Oren-Spedicato scaling factor}, or OS factor, as introduced by \cite{OS76}, i.e.,

\begin{equation}
\gamma_k = \frac{s_{k-1}^T y_{k-1}}{y_{k-1}^T y_{k-1}}.
\end{equation}

\noindent
Typically $H_0 = I$, i.e., $H_k$ is typically initialized as the identity matrix. The purpose of the OS factor is to dynamically adjust the initial Hessian approximation prior to the forward pass with the most recent curvature information. Doing so improves the convergence rate of L-BFGS (cf. \cite{GR23}). \\
\\
OS scaling also has the benefit of controlling the \textit{condition number} of $H_k$, i.e., $\kappa(H_k)$, which is crucial for ensuring the numerical stability of L-BFGS. For reference, note the formal definition of $\kappa(A)$ for some matrix $A$, courtesy of \cite{Kre98}.

\begin{definition}[Condition Number]
Let $X$ and $Y$ be normed spaces and let $A: X \rightarrow Y$ be a bounded linear operator with a bounded inverse $A^{-1}: Y \rightarrow X$. Then

\[
\kappa(A) := ||A|| \cdot ||A^{-1}||
\]

\noindent
where $\kappa(A)$ is the \textbf{condition number} of $A$.
\end{definition}

\noindent
While $\kappa(A)$ is a function of the particular norm chosen, it is nonetheless always the case that $\kappa(A) \in [1,\infty)$, where $\kappa(A) \approx 1$ corresponds to a well-conditioned matrix. For $A \in \mathbb{R}^{n \times n}$, in the $L^2$ norm, we have

\begin{equation}
\kappa_2(A) = \frac{|\lambda_{\text{ max}}|}{|\lambda_{\text{min}}|},
\end{equation}

\noindent
where $\lambda_{\text{min}}$ and $\lambda_{\text{max}}$ are the minimum and maximum eigenvalues of $A$, respectively.

\subsection{Contributions}

Despite its performance benefits on L-BFGS, OS scaling has the limitation that \textbf{OS scaling does not prevent $\kappa(H_k) \rightarrow \infty$}, which is more likely in non-convex loss landscapes. The crux of this paper is to propose a modified version of L-BFGS that we call \textit{Two-Sided L-BFGS}. We call it ``Two-Sided L-BFGS" because, for this version of L-BFGS, we restrict admissible $(s_k, y_k)$ into the memory buffer of $m$ vector pairs iff every $(s_k, y_k)$ satisfies the following two conditions:

\begin{equation}
\frac{y_k^T s_k}{||s_k||^2} \geq \epsilon,
\end{equation}

\noindent
for $\epsilon > 0$, and

\begin{equation}
\frac{||y_k||^2}{y_k^Ts_k} \leq M,
\end{equation}

\noindent
for $M \gg 0$. $[\epsilon, M]$ acts as a two-sided geometric envelope that constrains admissible $(s_k, y_k)$. The condition that $\frac{y_k^T s_k}{||s_k||^2} \geq \epsilon$ for $\epsilon > 0$ was introduced by \cite{LF01-JCAM} and is referred to as the \textit{cautious update rule} for L-BFGS. Quasi-Newton methods like (L-)BFGS require $B_k \succ 0$ to ensure line search methods (that satisfy the Wolfe conditions) produce descent directions, i.e.,

\[
y_k^T s_k > 0.
\]

\noindent
For convex functions, this condition is automatically satisfied under line searches that satisfy the Wolfe conditions. However, on non-convex landscapes, $y_k^T s_k$ can easily become zero or negative, hence the cautious update rule introduced by \cite{LF01-JCAM} shown in (15). If this condition fails, the update is skipped entirely (i.e., $B_{k+1} = B_k$), preventing numerical degradation and preserving $B_k \succ 0$. \\
\\
This paper extends upon \cite{LF01-JCAM} by also adding an upper bound on admissible $(s_k, y_k)$. While Li and Fukushima's cautious update rule acts as a lower bound that ensures that curvature does not vanish (i.e., $\lambda_{\min} \not\to 0$), the upper bound we propose in (16) prevents explosive curvature spikes (i.e., $\lambda_{\max} \not\to \infty$). In non-convex loss landscapes, a sudden, massive change in the norm of the gradient difference operator $\|y_k\|$ relative to $y_k^T s_k$ can cause $\kappa(H_k) \rightarrow \infty$. Enforcing this $[\epsilon, M]$ geometric envelope proactively prevents this. Note that $\frac{||y_k||^2}{y_k^T s_k} = \frac{1}{\gamma_k}$. Therefore, the upper bound condition is mathematically equivalent to enforcing a lower floor on the OS scaling parameter,

\begin{equation}
\gamma_k \geq \frac{1}{M}.
\end{equation}

\noindent
We show that incorporating this two-sided $[\epsilon, M]$ geometric envelope into Limited-Memory BFGS (L-BFGS) updates mitigates ill-conditioning better than relying on standard OS scaling alone. Our core (theoretical) contributions are as follows:

\begin{itemize}
    \item \textbf{Finite $\kappa(H_k)$ for Two-Sided L-BFGS:} In $\S$3, we show that, for any execution of Two-Sided L-BFGS at any inverse Hessian update $H_{k+1}$, $\kappa(H_{k+1}) \leq C < \infty$ for $C > 0$. Thus, Two-Sided L-BFGS has better overall numerical stability than that of L-BFGS with OS scaling alone.

    \item \textbf{Preserved Non-Convex L-BFGS Asymptotic Global Convergence:} \cite{LF01-SIAM} showed that L-BFGS with their cautious update rule guarantees that, for a general non-convex function with $\beta$-Lipschitz continuous gradients, the gradient sequence satisfies $\text{lim}_{k \rightarrow \infty} \text{inf}||\nabla f(x_k)|| = 0$ (i.e., cautious update L-BFGS never diverges and always converges to a stationary point). In $\S$4, we show that our Two-Sided L-BFGS preserves this non-convex asymptotic global convergence. 
\end{itemize}

\section{Two-Sided L-BFGS}

We present the pseudocode for our modified L-BFGS algorithm with a two-sided $[\epsilon, M]$ geometric envelope safeguarding mechanism as follows:

\begin{algorithm}[H]
\caption{Two-Sided L-BFGS}
\begin{algorithmic}
\Require Initialization $x_0$, memory depth $m$, tolerance $\tau > 0$, lower bound $\epsilon > 0$, upper bound $M \gg 0$, and $\gamma_0 = 1$
\State Compute $g_0 = \nabla f(x_0)$, set $k \gets 0$.
\State Initialize empty history queues $S = [\,]$, $Y = [\,]$, $\rho = [\,]$.
\While{$\|g_k\| > \tau$}
    \State Compute search direction $p_k$ using standard L-BFGS two-loop recursion with current history $\{S, Y, \rho\}$.
    \State Find step length $\alpha_k$ satisfying the strong Wolfe conditions (Inequalities (8) \& (10)).
    \State Update iterate: $x_{k+1} = x_k + \alpha_k p_k$.
    \State Compute $g_{k+1} = \nabla f(x_{k+1})$.
    \State Compute $s_k = x_{k+1} - x_k$ and $y_k = g_{k+1} - g_k$.
    \If{$\frac{y_k^T s_k}{\|s_k\|^2} \geq \epsilon$ \textbf{and} $\frac{\|y_k\|^2}{y_k^T s_k} \leq M$} \Comment{$[\epsilon, M]$ Two-Sided Envelope Condition}
    \If{$|S| == m$}
        \State Remove oldest vectors from $S, Y, \rho$.
    \EndIf
    \State Append $s_k$ to $S$, $y_k$ to $Y$, and $\rho_k = \frac{1}{y_k^T s_k}$ to $\rho$.
    \State $\gamma_{k+1} \gets \frac{s_k^T y_k}{\|y_k\|^2}$ \Comment{Compute and store OS factor filtered by $[\epsilon,M]$ envelope}
\Else
    \State \textbf{Skip Update:} Maintain existing history buffer
    \State $\gamma_{k+1} \gets \gamma_k$ \Comment{Carry forward the last known valid OS factor}
\EndIf
    \State $k \gets k + 1$
\EndWhile
\State \Return $x_k$
\end{algorithmic}
\end{algorithm}

\noindent
\textbf{Remark} Two-Sided L-BFGS preserves the $O(mn)$ space complexity and $O(mn)$ per-iteration time complexity of standard L-BFGS. Enforcing the $[\epsilon, M]$ two-sided envelope condition only requires computing basic vector dot products and norms already available or cheap to compute. This incurs an additional $O(n)$ computational overhead per iteration and $O(1)$ auxiliary memory, both of which are strictly dominated by the $O(mn)$ cost of the standard L-BFGS two-loop recursion.

\section{Upper Bound on Condition Number of Two-Sided L-BFGS Inverse Hessian Approximation Matrix}

Here, we show that, for any execution of Two-Sided L-BFGS at any inverse Hessian update $H_{k+1}$, $\kappa(H_{k+1}) \leq C < \infty$. Since $\kappa(H_{k+1}) = \frac{|\lambda_{\text{max}}|}{|\lambda_{\text{min}}|}$ by definition, we want to find $\frac{|\lambda_{\text{max}}|}{|\lambda_{\text{min}}|} \leq C$ for $C > 0$. We can do so via the Trace-Determinant lemma introduced by \cite{BN89}.

\begin{lemma}[Trace-Determinant Lemma, \cite{BN89}]
Let $H_{k+1}$ denote the updated inverse Hessian matrix through (L-)BFGS. Define the potential function

\[
\psi(H) := \text{tr}(H) - \text{ln}[\text{det}(H)],
\]

\noindent
where $\text{tr}(H)$ is the trace of $H$ and $\text{det}(H)$ is the determinant of $H$. If the curvature condition $y_k^T s_k > 0$ holds, then $\psi(H)$ satisfies the recursive inequality

\[
\psi(H_{k+1}) \leq \psi(H_k) + \frac{\|s_k \|^2}{y_k^T s_k} - 1 - \text{ln}(\frac{y_k^T s_k}{s_k^T B_k s_k}).
\]
\end{lemma}

\noindent
With the Trace-Determinant lemma in-hand, we can prove a finite upper bound on $\kappa(H_{k+1})$ as follows.

\begin{theorem}[Uniform Upper Bound on Two-Sided L-BFGS $\kappa(H_{k+1})$]
Let \(\{H_{k+1}\}_{k=0}^{\infty}\) be the sequence of inverse Hessian update approximation matrices generated by the Two-Sided L-BFGS algorithm (cf. Algorithm 3) with memory depth $m$, Li-Fukushima lower bound $\epsilon > 0$, and upper bound $M \gg 0$. Then, for any objective function \(f: \mathbb{R}^n \to \mathbb{R}\) and any iteration $k$, the condition number \(\kappa(H_{k+1})\) is uniformly bounded above by a finite constant \(\kappa _{\max }\) depending solely on the dimension \(n\), the memory depth \(m\), and the envelope parameters \(\epsilon \) and \(M\), i.e., 

\begin{equation} \kappa(H_{k+1}) \leq \kappa_{\max} < \infty, \quad \forall k \geq 0, 
\end{equation}

\noindent
where \(\kappa_{\max} = \frac{M_1}{m_1}\), and \(M_1, m_1\) are the positive, finite roots of the transcendental equation \(\lambda - \ln \lambda = C^*-n+1\), for any eigenvalue $\lambda$ of $H_{k+1}$, derived from the uniform upper bound \(C^*(n, m, \epsilon, M)\) of the Trace-Determinant lemma.
\end{theorem}

\begin{proof}
As in L-BFGS with standard OS scaling, at any iteration $k$, Two-Sided L-BFGS constructs $H_{k+1}$ via at most $m$ successive BFGS updates to the scaled identity matrix, i.e., $H_{k+1}^{(0)} = \gamma_{k+1}I$. By construction, Two-Sided L-BFGS only saves vector pairs $(s_k, y_k)$ at iteration $k$ to the memory buffer iff $\frac{y_k^T s_k}{\|s_k\|^2} \geq \epsilon$ for $\epsilon > 0$ and $\frac{\|y_k\|^2}{y_k^T s_k}$ for $M \gg 0$. Let the memory buffer of length $m$ be indexed by $j = 1,2, ..., \hat{m}$, where $\hat{m} \leq m$. Then the sequence of matrix updates in Two-Sided L-BFGS is

\begin{equation}
    \begin{cases}
        H_{k+1}^{(0)} = \gamma_{k+1}I \quad \text{for }j=0, \\
        H_{k+1}^{j} = (I - \rho_k s_jy_j^T)H_{k+1}^{(j-1)}(I - \rho_j y_j s_j^T) + \rho_j s_j s_j^T, \quad \text{for } j \in [1, \hat{m}].
    \end{cases}
\end{equation}

\noindent
The final matrix is $H_{k+1} = H_{k+1}^{(\hat{m})}$. By the Trace-Determinant lemma, we have

\begin{equation}
\psi(H_{k+1}^{(j)}) \leq \psi(H_{k+1}^{(j-1)}) + \frac{\|s_j\|^2}{y_j^T s_j} - 1 - \text{ln}(\frac{y_j^T s_j}{s_j^T B_{k+1}^{(j-1)}s_j}),
\end{equation}

\noindent
where $B_{k+1}^{(j-1)} = [H_{k+1}^{(j-1)}]^{-1}$. Summing inequality (20) telescopically for $j$ up to $\hat{m}$ yields

\begin{equation}
\psi(H_{k+1}) \leq \psi(\gamma_{k+1}I) + \sum_{j=1}^{\hat{m}}\frac{\|s_j\|^2}{y_j^T s_j} - \hat{m} - \sum_{j=1}^{\hat{m}}\text{ln}(\frac{y_j^T s_j}{s_j^T B_{k+1}^{(j-1)}s_j}).
\end{equation}

\noindent
Thus, 

\begin{equation}
\psi(H_{k+1}) \leq \psi(\gamma_{k+1}I) + \sum_{j=1}^{\hat{m}}\frac{\|s_j\|^2}{y_j^T s_j} - \sum_{j=1}^{\hat{m}}\text{ln}(\frac{y_j^T s_j}{s_j^T B_{k+1}^{(j-1)}s_j}),
\end{equation}

\noindent
or equivalently,

\begin{equation}
\psi(H_{k+1}) \leq \psi(\gamma_{k+1}I) + \sum_{j=1}^{\hat{m}}\frac{\|s_j\|^2}{y_j^T s_j} + \sum_{j=1}^{\hat{m}}\text{ln}(\frac{s_j^T B_{k+1}^{(j-1)}s_j}{y_j^T s_j}).
\end{equation}

\noindent
To upper bound $\psi(H_{k+1})$, it suffices to upper bound each of the three terms on the right-hand side of (23), respectively. \\
\\
We first find an upper bound on $\psi(\gamma_{k+1}I)$. Recall $\gamma_{k+1} = \frac{s_k^Ty_k}{\|y_k\|^2}$. By the ``upper side" of the envelope condition, $\frac{||y_k||^2}{y_k^Ts_k} \leq M$. Then $\gamma_{k+1} \geq \frac{1}{M}$. Likewise, the Li-Fukushima cautious update rule acts as the ``lower side" of the envelope condition, where $\frac{y_k^T s_k}{\|s_k\|^2} \geq \epsilon$. By the Cauchy-Schwarz inequality, $y_k^T s_k \leq \|y_k\| \|s_k\|$, which in turn implies $\|s_k\| \geq \frac{y_k^T s_k}{\|y_k\|}$. Substituting this inequality into the Li-Fukushima bound yields $\epsilon \leq \frac{y_k^T s_k}{\|s_k\|^2} \leq \frac{(y_k^T s_k)}{(\frac{y_k^T s_k}{\|y_k\|})^2} = \frac{\|y_k\|^2}{y_k^T s_k}$. Note that the Li-Fukushima bound also implies $\|s_k\| \leq \frac{\|y_k\|}{\epsilon}$. Then we have $\gamma_{k+1} = \frac{s_k y_k^T}{\|y_k\|^2} \leq \frac{\|s_k\| \|y_k\|}{\|y_k\|^2} = \frac{\|s_k\|}{\|y_k\|} \leq \frac{(\frac{\|y_k\|}{\epsilon})}{\|y_k\|} = \frac{1}{\epsilon}$, so $\gamma_{k+1} \leq \frac{1}{\epsilon}$. Therefore, the OS factor itself in Two-Sided L-BFGS is constrained such that $\gamma_{k+1} \in [\frac{1}{M}, \frac{1}{\epsilon}]$, so $\frac{1}{M} \leq \gamma_{k+1} \leq \frac{1}{\epsilon}$. By the Trace-Determinant lemma, $\psi(\gamma_{k+1}I) = \text{tr}(\gamma_{k+1}I) - \text{ln}[\text{det}(\gamma_{k+1})I] = n\gamma_{k+1} - n\text{ln}(\gamma_{k+1})$. Since the function $h(\gamma) = n\gamma - n\text{ln}(\gamma)$ is convex and continuous over $[\frac{1}{M}, \frac{1}{\epsilon}]$, $h(\gamma)$ is bounded above at the endpoints of $[\frac{1}{M}, \frac{1}{\epsilon}]$. Therefore,

\begin{equation}
\psi(\gamma_{k+1}I) \leq \text{max}\{\frac{n}{M}-n\text{ln}(\frac{1}{M}), \frac{n}{\epsilon}-n\text{ln}(\frac{1}{\epsilon})\} < \infty.
\end{equation}

\noindent
We next upper bound $\sum_{j=1}^{\hat{m}} \frac{\|s_j\|^2}{y_j^T s_j}$. By the Li-Fukushima bound, $\frac{\|s_j\|^2}{y_j^T s_j} \leq \frac{1}{\epsilon}$, so we have

\begin{equation}
\sum_{j=1}^{\hat{m}} \frac{\|s_j\|^2}{y_j^T s_j} \leq \frac{\hat{m}}{\epsilon} \leq \frac{m}{\epsilon} < \infty,
\end{equation}

\noindent
since $\hat{m} \leq m$. Lastly, we upper bound $\sum_{j=1}^{\hat{m}}\text{ln}(\frac{s_j^T B_{k+1}^{(j-1)}s_j}{y_j^T s_j})$. Consider $\frac{s_j^T B_{k+1}^{(j-1)}s_j}{y_j^T s_j}$. Note $\frac{s_j^T B_{k+1}^{(j-1)}s_j}{y_j^T s_j} = (\frac{s_j^T B_{k+1}^{(j-1)}s_j}{\|s_j\|^2})(\frac{\|s_j\|^2}{y_j^T s_j})$, so we can upper bound each of $\frac{s_j^T B_{k+1}^{(j-1)}s_j}{\|s_j\|^2}$ and $\frac{\|s_j\|^2}{y_j^T s_j}$, respectively. We already showed that $\frac{\|s_j\|^2}{y_j^T s_j} \leq \frac{1}{\epsilon}$ immediately follows from the Li-Fukushima bound. Since $B_{k+1}^{(j-1)}$ is symmetric, it follows from its Rayleigh quotient bound that $\frac{s_j^T B_{k+1}^{(j-1)}s_j}{\|s_j\|^2} \leq \lambda_{\text{max}}(B_{k+1}^{(j-1)})$. Now we want to upper bound $\lambda_{\text{max}}(B_{k+1}^{(j-1)})$. The update rule for $B_{k+1}^{(j-1)}$ is

\[
B^{(j)} = B^{(j-1)} - \frac{B^{(j-1)}s_j s_j^T B^{(j-1)}}{s_j^T B^{(j-1)}s_j} + \frac{y_j y_j^T}{y_j^T s_j}.
\]

\noindent
Taking the trace (Tr) of both sides yields

\[
\text{Tr}(B^{(j)}) = \text{Tr}(B^{(j-1)}) - \frac{\|B^{(j-1)}s_j\|^2}{s_j^T B^{(j-1)}s_j} + \frac{\|y_j\|^2}{y_j^T s_j}.
\]

\noindent
Since $- \frac{\|B^{(j-1)}s_j\|^2}{s_j^T B^{(j-1)}s_j} < 0$, this also gives

\[
\text{Tr}(B^{(j)}) \leq \text{Tr}(B^{(j-1)}) + \frac{\|y_j\|^2}{y_j^T s_j}.
\]

\noindent
But by the ``upper side" of the geometric envelope of Two-Sided L-BFGS, $\frac{\|y_j\|^2}{y_j^T s_j} \leq M$, so we then have

\[
\text{Tr}(B^{(j)}) \leq \text{Tr}(B^{(j-1)}) + M.
\]

\noindent
This inequality represents a telescoping sum from $j = 1$ to $\hat{m}$. For initialization $B^{(0)} = \frac{1}{\gamma_{k+1}}I$, this telescoping sum evaluates to

\[
\text{Tr}(B^{(\hat{m})}) \leq \text{Tr}(\frac{1}{\gamma_{k+1}}I) + \hat{m}M.
\]

\noindent
But recall we showed that $\gamma_{k+1} \geq \frac{1}{M}$, which implies $\frac{1}{\gamma_{k+1}} \leq M$, so $\text{Tr}(\frac{1}{\gamma_{k+1}}I) \leq nM$. Moreover, since $\hat{m} \leq m$, this gives us

\[
\text{Tr}(B^{(\hat{m})}) \leq nM + mM = (n+m)M.
\]

\noindent
Since $B^{(\hat{m})}$ is symmetric positive definite with $j \leq \hat{m}$, $\lambda_{\text{max}}(B_{k+1}^{(j-1)}) \leq \text{Tr}(B^{(\hat{m})})$. Thus,

\[
\lambda_{\text{max}}(B_{k+1}^{(j-1)}) \leq (n+m)M,
\]

\noindent
so $\frac{s_j^T B_{k+1}^{(j-1)}s_j}{\|s_j\|^2} \leq (n+m)M$. Then we yield

\[
\sum_{j=1}^{\hat{m}} \text{ln}(\frac{s_j^T B_{k+1}^{(j-1)}s_j}{y_j^T s_j}) \leq \sum_{j=1}^{\hat{m}} \text{ln}((n+m)M)(\frac{1}{\epsilon}) = \sum_{j=1}^{\hat{m}}\text{ln}(\frac{(n+m)M}{\epsilon}) = \hat{m}\text{ln}(\frac{(n+m)M}{\epsilon}) \leq m \text{ln}(\frac{(n+m)M}{\epsilon}).
\]

\noindent
Thus, we yield the upper bound

\begin{equation}
\sum_{j=1}^{\hat{m}}\text{ln}(\frac{s_j^T B_{k+1}^{(j-1)}s_j}{y_j^T s_j}) \leq m\text{ln}(\frac{(n+m)M}{\epsilon}) < \infty.
\end{equation}

\noindent
Then summing inequalities (24), (25), and (26) gives

\begin{equation}
\psi(H_{k+1}) \leq \text{max}\{\frac{n}{M}-n\text{ln}(\frac{1}{M}), \frac{n}{\epsilon}-n\text{ln}(\frac{1}{\epsilon})\} + \frac{m}{\epsilon} + m\text{ln}(\frac{(n+m)M}{\epsilon}) < \infty.
\end{equation}

\noindent
Let $C^{\star}(n,m, \epsilon, M) := \text{max}\{\frac{n}{M}-n\text{ln}(\frac{1}{M}), \frac{n}{\epsilon}-n\text{ln}(\frac{1}{\epsilon})\} + \frac{m}{\epsilon} + m\text{ln}(\frac{(n+m)M}{\epsilon})$, so we likewise have

\begin{equation}
\psi(H_{k+1}) \leq C^{\star}(n,m, \epsilon, M) < \infty.
\end{equation}

\noindent
Byrd and Nocedal's potential function for $H_{k+1}$ can likewise be written as

\begin{equation}
\psi(H_{k+1}) = \sum_{i=1}^{n}(\lambda_i - \text{ln}\lambda_i)
\end{equation}

\noindent
for eigenvalues $\lambda_i$ of $H_{k+1}$. For any arbitrary eigenvalue $\lambda_k$ of $H_{k+1}$, we have

\begin{equation}
\psi(H_{k+1}) = (\lambda_k - \text{ln}\lambda_k) + \sum_{i \neq k}^{n}(\lambda_i - \text{ln}\lambda_i).
\end{equation}

\noindent
$h(\lambda) = \lambda - \text{ln}\lambda$ has a global minimum at $\lambda = 1$, and $h(1) = 1$. Thus, $\sum_{i \neq k}^{n}(\lambda_i - \text{ln}\lambda_i) \geq (n-1)(1) = n-1$. This inequality and Eqn. (28) together imply

\begin{equation}
\lambda_k - \text{ln}\lambda_k \leq \psi(H_{k+1}) - n + 1.
\end{equation}

\noindent
Since $\psi(H_{k+1}) \leq C^\star$, we equivalently have

\begin{equation}
\lambda_k - \text{ln}\lambda_k \leq C^\star - n + 1.
\end{equation}

\noindent
Because the function $h(\lambda) = \lambda - \ln\lambda$ satisfies $\lim_{\lambda \to 0^+} h(\lambda) = \infty$ and $\lim_{\lambda \to \infty} h(\lambda) = \infty$, the transcendental equation $\lambda - \ln\lambda = C^* - n + 1$ possesses exactly two positive roots, $m_1$ and $M_1$ where $m_1 \leq M_1$. Therefore, for all eigenvalues $\lambda_i$ of $H_{k+1}$, $\lambda_i(H_{k+1}) \in [m_1, M_1]$. Thus,

\begin{equation}
0 < m_1 \leq \lambda_i(H_{k+1}) \leq M_1 < \infty, \quad \forall i \in [1,n].
\end{equation}

\noindent
Hence, $\kappa(H_{k+1})$ is uniformly bounded by

\begin{equation}
\kappa(H_{k+1}) = \frac{\lambda_{\text{max}}(H_{k+1})}{\lambda_{\text{min}}(H_{k+1})} \leq \frac{M_1}{m_1} = \kappa_{\text{max}}(H_{k+1}) < \infty.
\end{equation}
\end{proof}

\noindent
\textbf{Remark} Theorem 6 proves that Two-Sided L-BFGS ensures $\kappa(H_{k+1}) \not\to \infty$. Moreover, recall $C^{\star}(n,m, \epsilon, M) := \text{max}\{\frac{n}{M}-n\text{ln}(\frac{1}{M}), \frac{n}{\epsilon}-n\text{ln}(\frac{1}{\epsilon})\} + \frac{m}{\epsilon} + m\text{ln}(\frac{(n+m)M}{\epsilon})$. One can verify that $C^\star$ is $O(\frac{n}{\epsilon} + m\text{ln}n)$ given the conditions $n \gg M \gg m$ and small $\epsilon > 0$. For sufficiently large $n$, $C^\star$ is $O(\frac{n} {\epsilon})$. Theorem 6 ensures that for any fixed dimension $n$, the condition number is uniformly bounded above by a finite constant ($\kappa(H_{k+1}) \leq \kappa_{\text{max}} < \infty$), guaranteeing that the matrix never becomes strictly singular or infinitely ill-conditioned during the execution of Two-Sided L-BFGS. However, because the trace-determinant potential function bounds the spectrum via a transcendental equation, this theoretical worst-case bound $\kappa_{\text{max}}$ degrades exponentially with $n$, scaling $O(n e^n)$.

\section{Convergence Analysis of Two-Sided L-BFGS}

\cite{LF01-SIAM} showed that L-BFGS with the Li-Fukushima cautious update rule has asymptotic global convergence in the non-convex regime. While Two-Sided L-BFGS, as we showed in $\S$3, has improved numerical stability over that of standard L-BFGS with OS scaling (since Two-Sided L-BFGS ensures $\kappa(H_{k+1}) \not\rightarrow \infty$), it begs the question as to whether Two-Sided L-BFGS preserves this non-convex asymptotic global convergence. Here, we show that this is indeed the case. To do so, we require $f: \mathbb{R}^n \to \mathbb{R}$ to satisfy the existence of a bounded level set, i.e.,

\begin{definition}[Bounded Level Set]
Let $f: \mathbb{R}^n \to \mathbb{R}$. Then $\exists \Omega \in \text{dom}(f)$, called a \textbf{level set}, where

\[
\Omega = \{x \in \mathbb{R}^n \mid f(x) \leq f(x_0)\}
\]

\noindent
and $\Omega$ is bounded.
\end{definition}

\noindent
This ensures the sequence $\{x_k\}$ stays within a compact region, meaning $f$ is bounded below, i.e.,  $f(x) \geq f^* > -\infty$. We also require $f$ to have $\beta$-Lipschitz gradients, i.e.,

\begin{definition}[$\beta$-Lipschitz Gradient]
Let $f: \mathbb{R}^n \to \mathbb{R}$. If $f$ has \textbf{$\beta$-Lipschitz gradients}, then $\exists \beta > 0$ such that, $\forall x,y \in \Omega$,

\[
\|\nabla f(x) - \nabla f(y)\| \leq \beta \|x - y\|.
\]
\end{definition}

\noindent
If these two conditions hold, then, as we show, Two-Sided L-BFGS preserves non-convex asymptotic global convergence.

\begin{theorem}[Asymptotic Global Convergence of Two-Sided L-BFGS]
Let $f: \mathbb{R}^n \to \mathbb{R}$ be continuously differentiable on $\mathbb{R}^n$. Suppose the Two-Sided L-BFGS algorithm (cf. Algorithm 3) generates an infinite sequence of iterates $\{x_k\}_{k=0}^{\infty}$ using a step length $\alpha_k$ that satisfies the strong Wolfe conditions:

\[
f(x_k + \alpha_k p_k) \leq f(x_k) + c_1 \alpha_k \nabla f(x_k)^T p_k,
\]

\[
|\nabla f(x_k + \alpha_k p_k)^T p_k| \leq c_2 |\nabla f(x_k)^T p_k|,
\]

\noindent
where $0 < c_1 < c_2 < 1$. Let the historical updates $(s_k, y_k)$ be conditionally accepted into the memory buffer of depth $m$ under the $[\epsilon, M]$ envelope

\[
\frac{y_k^T s_k}{\|s_k\|^2} \geq \epsilon \quad \text{and} \quad \frac{\|y_k\|^2}{y_k^T s_k} \leq M,
\]

\noindent
for constants $\epsilon > 0$ and $M \gg 0$. If the following standard conditions hold,
\begin{enumerate}
\item The initial level set $\Omega = \{x \in \mathbb{R}^n \mid f(x) \leq f(x_0)\}$ is bounded,
\item The gradient $\nabla f$ is $\beta$-Lipschitz continuous on $\Omega$,
\end{enumerate}

\noindent
then Two-Sided L-BFGS converges globally to a first-order stationary point, i.e., 

\[
\lim_{k \to \infty} \|\nabla f(x_k)\| = 0.
\]
\end{theorem}

\begin{proof}
Since Two-Sided L-BFGS finds step length $\alpha_k$ satisfying the strong Wolfe conditions at any iteration $k$, the Armijo condition holds, so we have

\begin{equation}
f(x_{k+1}) \leq f(x_k) + c_1\alpha_k \nabla f(x_k)^Tp_k,
\end{equation}

\noindent
for $c_1 \in (0,1)$. Since $B_k \succ 0$ by construction, search directions produced by Two-Sided L-BFGS satisfy

\[
p_k = -H_k \nabla f(x_k).
\]

\noindent
Moreover, since $H_k = H_k^T \succ 0$ by Theorem 6, $p_k$ is a descent direction so $\nabla f(x_k)^T p_k = -\nabla f(x_k)^T H_k \nabla f(x_k) < 0$. Let $\theta_k$ define the angle between $p_k$ and the steepest descent direction $-\nabla f(x_k)$. This implies

\begin{equation}
\text{cos}\theta_k = \frac{-\nabla f(x_k)^T p_k}{\|\nabla f(x_k)\| \|p_k\|}.
\end{equation}

\noindent
This means that the Armijo condition can likewise be expressed as

\[
f(x_{k+1}) \leq f(x_k) - c_1 \alpha_k \text{cos}\theta_k \|\nabla f(x_k)\| \|p_k\|.
\]

\noindent
We next lower bound $\alpha_k$. By the second Wolfe condition, we have

\[
\nabla f(x_{k+1})^T p_k \geq c_2 \nabla f(x_k)^T p_k,
\]

\noindent
for $c_2 \in (c_1, 1)$. By algebraic manipulation,

\[
\nabla f(x_{k+1})^T p_k - \nabla f(x_k)^T p_k \geq c_2 \nabla f(x_k)^Tp_k - \nabla f(x_k)^T p_k,
\]

\[
(\nabla f(x_{k+1})-\nabla f(x_k))^T p_k \geq (c_2 - 1)\nabla f(x_k)^Tp_k.
\]

\noindent
By the Cauchy-Schwarz inequality, $(\nabla f(x_{k+1})-\nabla f(x_k))^T p_k \leq \|\nabla f(x_{k+1})-\nabla f(x_k)\| \|p_k\|$. Moreover, since $f$ has $\beta$-Lipschitz gradients, $\|\nabla f(x_{k+1})-\nabla f(x_k)\| \leq \beta \|x_{k+1} - x_k \|$. Thus, we can infer

\[
\beta \|x_{k+1} - x_k\| \geq (c_2 - 1)\nabla f(x_k)^T p_k,
\]

\noindent
or equivalently,

\[
\beta \|x_{k+1} - x_k\| \geq (1-c_2)(-\nabla f(x_k)^T p_k).
\]

\noindent
Two-Sided L-BFGS performs parameter updates via the standard L-BFGS $x_{k+1} = x_k - \alpha_k p_k$ update rule. One can verify by computation that $\beta \|x_{k+1} - x_k\| = \beta \alpha_k \|p_k\|$. Then we have

\[
\beta \alpha_k \|p_k\|^2 \geq (1-c_2)(-\nabla f(x_k)^T p_k).
\]

\noindent
By computation, we lower bound $\alpha_k$,

\[
\alpha_k \geq \frac{(1-c_2)(-\nabla f(x_k)^T p_k)}{\beta \|p_k\|^2},
\]

\[
\alpha_k \geq (\frac{1-c_2}{\beta})(\frac{-\nabla f(x_k)^T p_k}{\|p_k \|^2}),
\]

\[
\alpha_k \geq (\frac{1-c_2}{\beta})(\frac{\|\nabla f(x_k)\| \|p_k\| \text{cos}\theta_k}{\|p_k\|^2}),
\]

\noindent
so we have the following lower bound on $\alpha_k$:

\begin{equation}
\alpha_k \geq (\frac{1-c_2}{\beta})(\frac{\| \nabla f(x_k)\| \text{cos}\theta_k}{\|p_k\|}).
\end{equation}

\noindent
Then substituting the lower bound on $\alpha_k$ in (37) into the Armijo condition implies 

\[
f(x_{k+1}) \leq f(x_k) - c_1(\frac{1-c_2}{\beta})(\frac{\|\nabla f(x_k)\| \text{cos}\theta_k}{\|p_k\|})\text{cos}\theta_k \|\nabla f(x_k\| \|p_k\|,
\]

\[
f(x_{k+1}) \leq f(x_k) - [\frac{c_1(1-c_2)}{\beta}]\| \nabla f(x_k)\|^2 \text{cos}^2\theta_k.
\]

\noindent
We show that $\sum_{k=0}^{\infty} \|\nabla f(x_k)\|^2 \text{cos}^2\theta_k < \infty$. By computation,

\[
f(x_{k+1}) - f(x_k) \leq -[\frac{c_1(1-c_2)}{\beta}]\| \nabla f(x_k) \| \text{cos}^2\theta_k,
\]

\[
\|\nabla f(x_k)\|^2 \text{cos}^2\theta_k \leq - [\frac{\beta}{c_1(1-c_2)}][f(x_{k+1}) - f(x_k)],
\]

\[
\|\nabla f(x_k)\|^2 \text{cos}^2\theta_k \leq [\frac{\beta}{c_1(1-c_2)}][f(x_k) - f(x_{k+1})],
\]

\[
[\frac{c_1(1-c_2)}{\beta}]\|\nabla f(x_k)\|^2 \text{cos}^2\theta_k \leq f(x_k) - f(x_{k+1}).
\]

\noindent
Suppose we run Two-Sided L-BFGS for $T$ iterations. Then from $k=0$ to $k=T$ we have

\[
[\frac{c_1(1-c_2)}{\beta}]\sum_{k=0}^{T}\|\nabla f(x_k)\|^2 \text{cos}^2\theta_k \leq \sum_{k=0}^{T}(f(x_k) - f(x_{k+1})).
\]

\noindent
$\sum_{k=0}^{T}(f(x_k) - f(x_{k+1}))$ is a telescoping series, which implies $\sum_{k=0}^{T}(f(x_k) - f(x_{k+1})) = f(x_0) - f(x_{T+1})$, so we have

\[
[\frac{c_1(1-c_2)}{\beta}]\sum_{k=0}^{T}\|\nabla f(x_k)\|^2 \text{cos}^2\theta_k \leq f(x_0) - f(x_{T+1}).
\]

\noindent
Thus, 

\begin{equation}
\sum_{k=0}^{T}\|\nabla f(x_k)\|^2 \text{cos}^2\theta_k \leq [\frac{\beta}{c_1(1-c_2)}](f(x_0) - f(x_{T+1})).
\end{equation}

\noindent
Since $\Omega$ is a bounded level set, $f$ is bounded below by $f^*$. Taking the limit as $T \to \infty$, the right side of (38) remains bounded, so we obtain

\begin{equation}
\sum_{k=0}^{\infty} \|\nabla f(x_k)\|^2 \cos^2\theta_k < \infty. 
\end{equation}

\noindent
Inequality (39) is known as the \textit{Zoutendijk condition} (cf. \cite{Zou60}), so we have shown that Two-Sided L-BFGS preserves the Zoutendijk condition. To show that the gradient norm series must vanish, we require that the search directions do not become asymptotically orthogonal to the gradient vector (i.e., $\cos^2\theta_k$ must be uniformly bounded away from zero). Since $H_k = H_k^T \succ 0$ and $\kappa_{\max}(H_{k+1}) = \frac{M_1}{m_1} < \infty$ by Theorem 6, the eigenvalues of $H_k$ are strictly bounded within the positive compact interval $[m_1, M_1]$. This implies

\begin{equation}
\|H_k \nabla f(x_k)\| \leq M_1 \|\nabla f(x_k)\| \quad \text{and} \quad \nabla f(x_k)^T H_k \nabla f(x_k) \geq m_1 \|\nabla f(x_k)\|^2.
\end{equation}

\noindent
Since $p_k = -H_k \nabla f(x_k)$ and $\text{cos}\theta_k = \frac{-\nabla f(x_k)^T p_k}{\|\nabla f(x_k)\| \|p_k\|}$, we have

\begin{equation}
\text{cos}\theta_k = \frac{\nabla f(x_k)^T H_k \nabla f(x_k)}{\|\nabla f(x_k)\| \|p_k\|}.
\end{equation}

\noindent
Substituting the bounds in (40) into (41) implies

\begin{equation}
\cos\theta_k = \frac{\nabla f(x_k)^T H_k \nabla f(x_k)}{\|\nabla f(x_k)\| \|H_k \nabla f(x_k)\|} \geq \frac{m_1 \|\nabla f(x_k)\|^2}{\|\nabla f(x_k)\| \cdot M_1 \|\nabla f(x_k)\|} = \frac{m_1}{M_1} = \frac{1}{\kappa_{\max}} > 0,
\end{equation}

\noindent
which in turn implies

\begin{equation}
\text{cos}^2\theta_k \geq \frac{1}{\kappa^2_{\text{max}}} > 0, \quad \forall k \geq 0.
\end{equation}

\noindent
(43) shows that search directions generated by Two-Sided L-BFGS are uniformly bounded away from orthogonality to the gradient vector. Substituting this uniform lower bound into the Zoutendijk condition (inequality (39)) yields

\begin{equation}
\sum_{k=0}^{\infty} \|\nabla f(x_k)\|^2 \left(\frac{1}{\kappa_{\max}^2}\right) \leq \sum_{k=0}^{\infty} \|\nabla f(x_k)\|^2 \cos^2\theta_k < \infty,
\end{equation}

\noindent
so we likewise have

\begin{equation}
\frac{1}{\kappa_{\max}^2} \sum_{k=0}^{\infty} \|\nabla f(x_k)\|^2 \leq \sum_{k=0}^{\infty} \|\nabla f(x_k)\|^2 \cos^2\theta_k < \infty.
\end{equation}

\noindent
Thus,

\begin{equation}
\lim_{k \to \infty} \|\nabla f(x_k)\|^2 = 0 \implies \lim_{k \to \infty} \|\nabla f(x_k)\| = 0.
\end{equation}

\noindent
Therefore, the sequence of iterates generated by the Two-Sided L-BFGS algorithm converges globally to a first-order stationary point in the non-convex regime.
\end{proof}

\section{Numerical Experiments}

The theoretical analysis in $\S$3 and $\S$4 establishes that introducing a two-sided $[\epsilon, M]$ geometric envelope safeguards the inverse Hessian approximation matrix from becoming infinitely ill-conditioned ($\kappa(H_{k+1}) \leq \kappa_{\max} < \infty$) while fully preserving non-convex asymptotic global convergence. The objective of these experiments is to demonstrate that this mathematical safeguarding translates directly into improved numerical stability, robust line-search performance, and accelerated wall-clock convergence on highly non-convex, ill-conditioned optimization landscapes. To do so, we present a set of three different numerical experiments:
\begin{enumerate}
    \item \textbf{Controlled Stress-Testing on Classical Non-Convex Functions} We evaluate Two-Sided L-BFGS on high-dimensional variants of the Rosenbrock function to explicitly track the trajectory of the condition number $\kappa(H_k)$ and observe how the geometric envelope mitigates the characteristic stagnation or "flattening" of the optimality gap, $|| \nabla f(x_k) ||$, common to standard L-BFGS under extreme ill-conditioning.
    \item \textbf{Standard Unconstrained Optimization Benchmarks} Utilizing challenging instances from the CUTEst benchmark suite (such as the notoriously ill-conditioned \texttt{DIXMAANA} problem), we measure search-direction degradation via an orthogonality metric, $\cos\theta_k$. We further analyze the efficiency of the line-search mechanism by tracking the distribution of function evaluations per iteration to determine if the safeguarded matrix spectrum eliminates costly backtracking loops.
    \item \textbf{Large-Scale Deep Learning Applications} We deploy the algorithm to train deep autoencoders on MNIST. Because deep autoencoder loss landscapes are inherently non-convex and prone to severe numerical instability, this serves as a benchmark for evaluating wall-clock efficiency, confirming that the $O(n)$ safeguarding overhead is strictly dominated by the $O(mn)$ two-loop recursion.
\end{enumerate}

\noindent
All algorithms are implemented in Python with standard libraries (e.g., NumPy, PyTorch) and executed on a Macbook with a 1.4 GHz Quad-Core Intel Core i5 processor to ensure a fair and consistent baseline for wall-clock comparisons. In all benchmarks, standard L-BFGS with Oren-Spedicato (OS) scaling serves as our primary baseline.

\subsection{Rosenbrock Function}

We simulate the optimization of a high-dimensional Rosenbrock function ($n = 100$) using standard L-BFGS with Oren-Spedicato (OS) scaling versus our proposed Two-Sided L-BFGS. The high-dimensional Rosenbrock function is notoriously ill-conditioned and non-convex, defined as

\begin{equation}
f(x) = \sum_{i=1}^{n-1} \left[ 100(x_{i+1} - x_i^2)^2 + (1 - x_i)^2 \right].
\end{equation}

\noindent
The experimental setup is as follows: set dimension $n=100$ and memory depth $m=10$, and initialization $x_0 = [-1.2, 1.0, -1.2, 1.0, ...]$. For Two-Sided L-BFGS, we use envelope hyperparameter values $\epsilon = 10^{-4}$ and $M = 10^4$. Both algorithms deploy an identical strong Wolfe line search with $c_1 = 10^{-4}$ and $c_2 = 0.9$. See Figure 1 for simulation results.

\begin{figure}[H]
    \centering
    \includegraphics[width=0.8\linewidth]{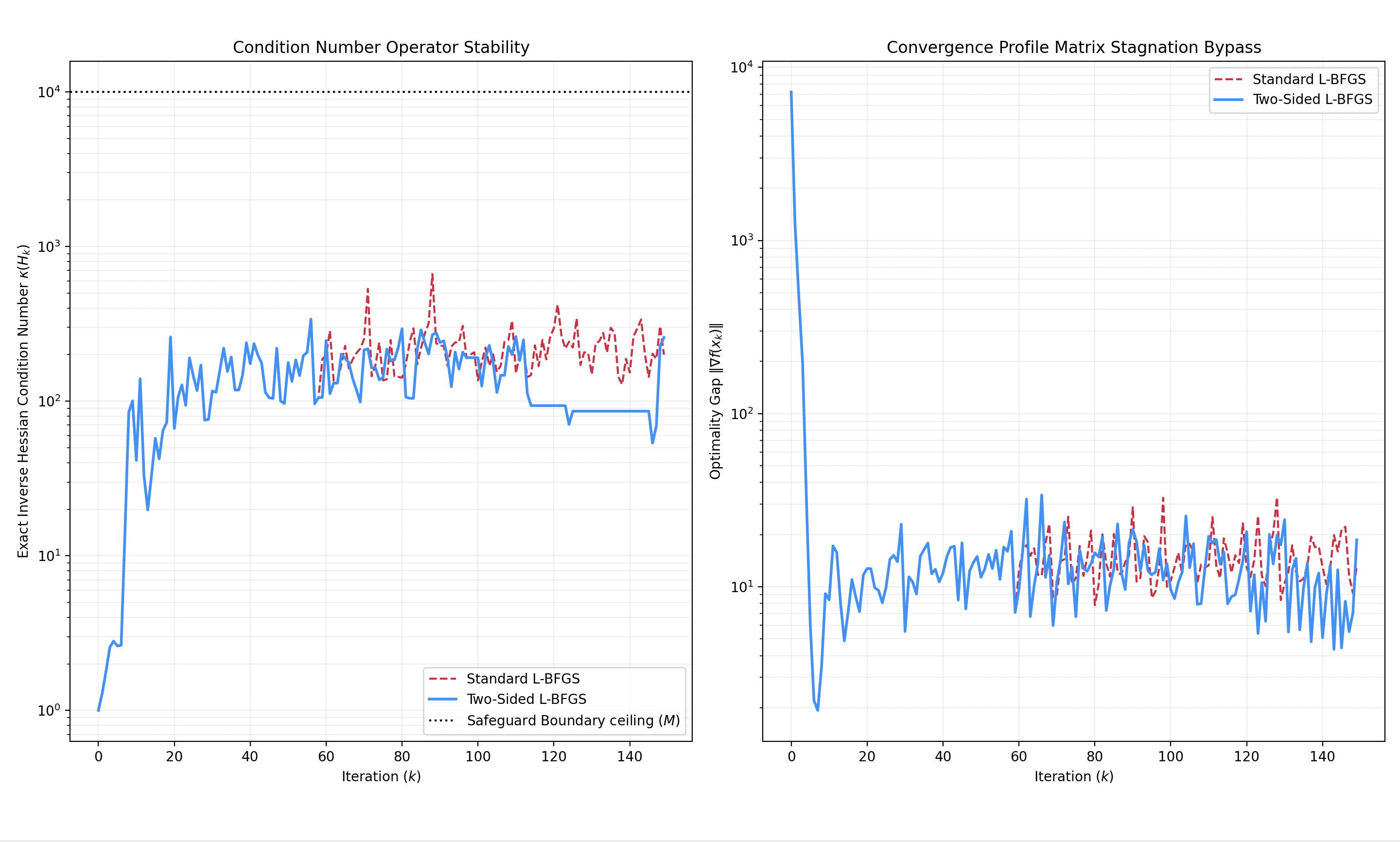}
    \caption{$\kappa(H_k)$ and $\| \nabla f(x_k)\|$ over $k$ in Two-Sided L-BFGS vs. Standard L-BFGS with OS Scaling}
\end{figure}

\subsection{Modified DIXMAAN Benchmark}

We utilize a truncated DIXMAAN benchmark function of dimension $n$ that takes the form

\begin{equation}
f(x) = 1 + \sum_{i=1}^n \alpha x_i^2 \left( \frac{i}{n} \right)^{k_1} + \sum_{i=1}^{n-1} \beta x_i^2 (x_{i+1} + x_{i+1}^2)^2 \left( \frac{i}{n} \right)^{k_2},
\end{equation}

\noindent
where we set $\alpha = 1.0, \beta = 1.0, k_1 = 2, k_2 = 2$, which are set to induce ill-conditioning. See Figure 2 for simulation results.

\begin{figure}[H]
    \centering
    \includegraphics[width=1\linewidth]{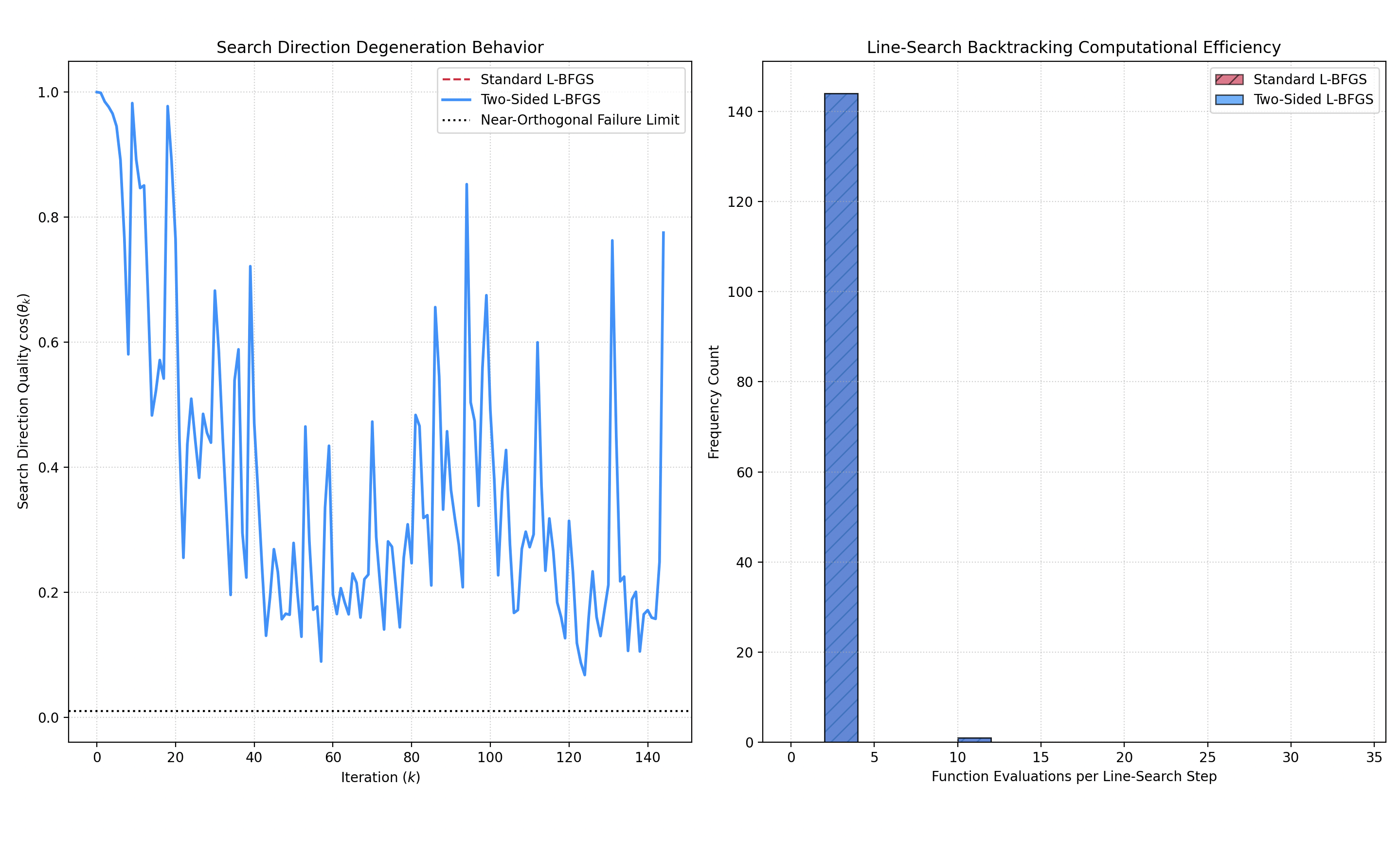}
    \caption{$\text{cos}\theta_k$ over $k$ and Number of Function Evaluations per Line-Search Step in Two-Sided L-BFGS vs. Standard L-BFGS with OS scaling}
\end{figure}

\subsection{Deep Autoencoder on MNIST}

For this experiment, we evaluate the proposed Two-Sided L-BFGS algorithm against standard L-BFGS on a deep autoencoder training task using the MNIST dataset. Deep autoencoders are well-known for creating highly non-convex, ill-conditioned optimization landscapes (cf. \cite{HS06}). Because the bottleneck layer forces a massive compression and subsequent expansion of data, the gradient backpropagation often yields directions with highly erratic scales. This acts as an excellent environment for testing any quasi-Newton method's true structural robustness, in this case, our proposed Two-Sided L-BFGS algorithm. One possible pushback against deploying the $[\epsilon, M]$ geometric envelope safeguard is that the memory overhead required to constrain admissible $(s_k, y_k)$ vector pairs destroys the algorithm's actual runtime efficiency. In our theoretical analysis, we showed that the computational complexity of checking the envelope condition $\frac{y_k^T s_k}{\|s_k\|^2} \geq \epsilon$ and $\frac{\|y_k\|^2}{y_k^T s_k} \leq M$ requires only a few vector dot products, which scales $O(n)$, so the $[\epsilon , M]$ geometric envelope should only introduce negligible wall-clock latency per iteration. We conjecture that Two-Sided L-BFGS will have improved wall-clock efficiency compared to standard L-BFGS. Our autoencoder consists of seven layers in a $784 \to 128 \to 64 \to 32 \to 64 \to 128 \to 784$ neuron by-layer arrangement for compression on the original $28 \times 28$ images in the MNIST dataset. See Figure 3 for the simulation results.

\begin{figure}[H]
    \centering
    \includegraphics[width=1\linewidth]{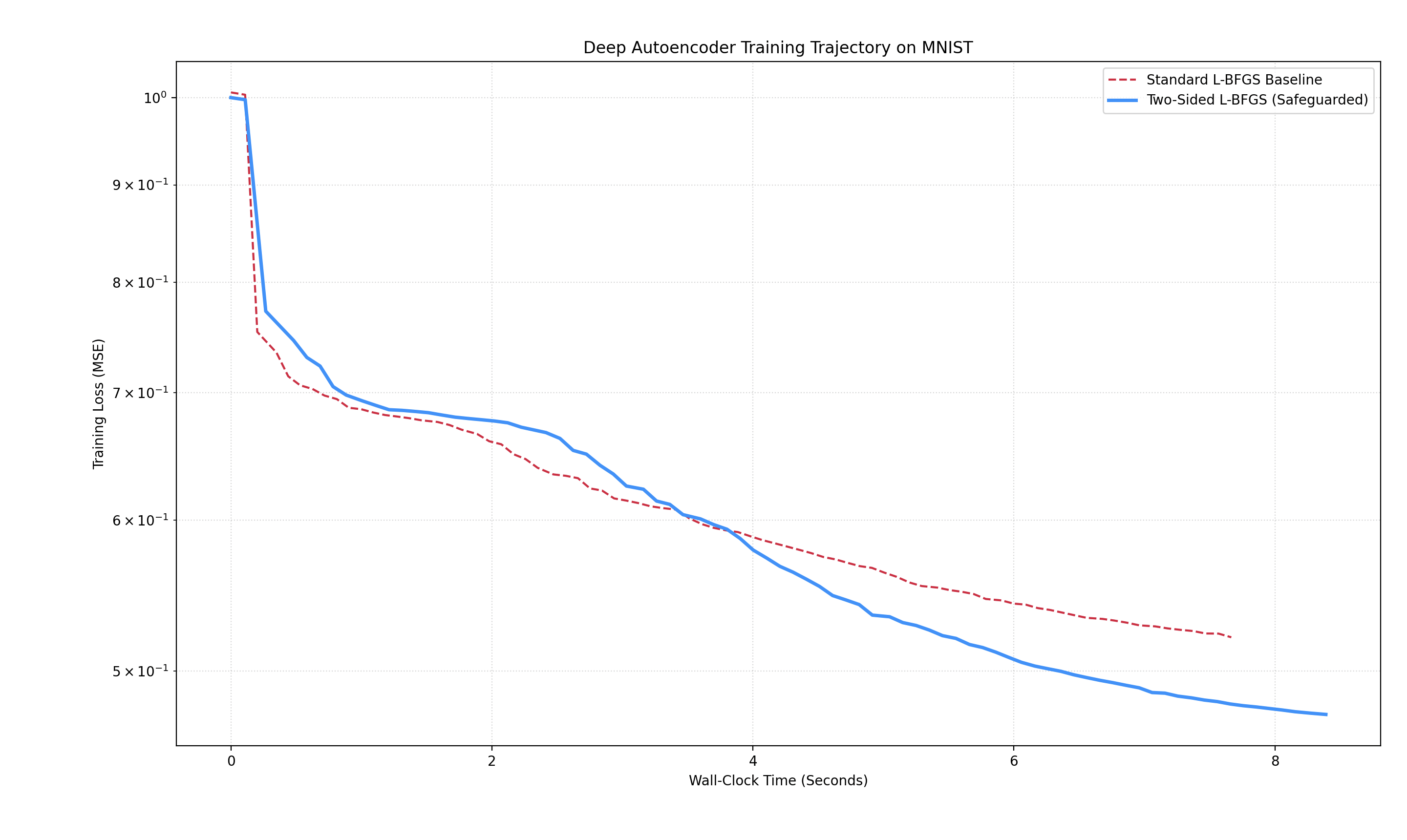}
    \caption{Wall-Clock Efficiency of Two-Sided L-BFGS vs. Standard L-BFGS in Deep Autoencoder Training on MNIST}
\end{figure}

\section{Discussion}

The theoretical guarantees and empirical validations presented in this work provide a unified justification for the proposed Two-Sided L-BFGS algorithm. We reflect on our results and consider some future research directions. \\
\\
A major pathology of standard L-BFGS in non-convex optimization is its vulnerability to severe ill-conditioning. Two-Sided L-BFGS directly addresses this with the $[\epsilon, M]$ geometric envelope by uniformly ensuring $\kappa(H_{k+1}) \not\to \infty$ (cf. Theorem 6), i.e., flattening $\kappa(H_{k+1})$ strictly beneath the deterministic upper ceiling dictated by the positive roots of our transcendental bounding equation. Rather than allowing the inverse Hessian to degrade and force numerical overflow (e.g., NaN errors), which caused the collapse of the standard baseline, as manifested in the failure of standard L-BFGS on the truncated DIXMAAN benchmark in experiment 2, Two-Sided L-BFGS preserves a structurally stable inverse Hessian operator throughout the entire optimization trajectory. Moreover, experiment 2 shows that the vast majority of function evaluations per line-search step in Two-Sided L-BFGS on the truncated DIXMAAN benchmark is tightly clustered around 2 to 4 function evaluations per line-search step. The spectral boundaries enforced by Two-Sided L-BFGS on $H_{k+1}$ yield a guaranteed lower cushion on search direction alignment, proving mathematically that $\cos^2\theta_k \geq 1/\kappa_{\max}^2 > 0$. Experiment 2 empirically validates this result, showing that the search direction maintains an effective alignment away from orthogonality. Consequently, the line search rapidly accepts large step lengths, frequently terminating within 2 to 4 evaluations. By bypassing line-search bottlenecks, Two-Sided L-BFGS avoids the characteristic flattening of the optimization curve and satisfies the Zoutendijk condition established in Theorem 7. \\
\\
A possible critique of algorithmic modifications that add safeguarding conditions is that the additional computational complexity degrades real-time efficiency. However, our deep autoencoder benchmarks on MNIST (Experiment 3) demonstrate that the safeguarding mechanism is quite lightweight. From a complexity perspective, evaluating the two-sided envelope condition requires only standard vector dot products ($\|s_k\|^2$, $\|y_k\|^2$, and $y_k^T s_k$), which scales strictly as $O(n)$. Because this operations cost is significantly smaller than the $O(mn)$ complexity required by the standard L-BFGS two-loop recursion, the per-iteration wall-clock cost remains virtually unchanged. The true value of this design choice is visible in the total wall-clock trajectory. While standard L-BFGS wastes computational time backtracking through deformed search directions, Two-Sided L-BFGS uses its well-conditioned operator to achieve smoother, more rapid descent. The slight cost of checking the envelope yields a substantial return in global optimization speed, allowing the model to reach target training losses in less total wall-clock time required by unconstrained baselines. \\
\\
Despite the promising theoretical and empirical results, the performance of Two-Sided L-BFGS depends on the proper selection of the envelope hyperparameters $\epsilon$ and $M$. If the lower bound $\epsilon$ is chosen too aggressively or the upper bound $M$ is made too restrictive, the algorithm becomes overly defensive. In this scenario, it may skip a high frequency of valid curvature updates, causing the inverse Hessian approximation to reduce back toward identity scaling and slowing convergence to that of standard gradient descent. A promising direction for future research lies in developing an adaptive or dynamical envelope. By automatically scaling $\epsilon_k$ and $M_k$ based on local gradient variance or tracking the recent frequency of skipped updates, the algorithm could dynamically tune its defensiveness to match the volatility of the local landscape. For example, a possible modification of Two-Sided L-BFGS is to make the envelope hyperparameters dynamical hyperparameters $\epsilon_k$, $M_k$ such that

\[
\epsilon_k = \epsilon_0 \cdot \text{min}(1, \|\nabla f(x_k)\|,
\]

\[
M_k = M_0 \cdot \text{max}(1, \|\nabla f(x_k)\|.
\]

\noindent
Additionally, here we assumed a deterministic, full-batch optimization setting. Extending the two-sided geometric envelope to stochastic regimes, such as mini-batch training for deep neural networks, presents a compelling open challenge. In stochastic environments, the noise inherent in mini-batch gradient estimations could easily cause false envelope violations, prematurely triggering update skips. Overcoming this will require integrating the two-sided envelope with stochastic variance-reduction techniques or running moving-average smoothing over the history pairs to separate structural curvature shifts from random mini-batch noise.
\bibliography{kappaL-BFGS}
\end{document}